%% file: main.tex
\documentclass[a4paper, reqno]{amsart}

\input{preamble}

\begin{document}

\nolinenumbers

\title{\afbf{Lie methods for countably categorical Engel groups: the Wilson conjecture for $4$-Engel $5$-groups}}

\author[C. d'Elb\'{e}e]{Christian d\textquoteright Elb\'ee}
\address{University of the Basque Country, Department of Mathematics (Leioa) / Institute for Logic, Cognition, Language and Information (Donostia-San Sebastián), Basque Country, Spain}
\email{christian.delbee@ehu.eus}
\urladdr{\href{http://choum.net/\textasciitilde chris/page\textunderscore perso/}{http://choum.net/\textasciitilde chris/page\textunderscore perso/}}

\thanks{
The author is supported by the Basque Government grant reference number IT1913-26, and by the Ramon y Cajal grant RYC2023-042677-I funded by MICIU/AEI/10.13039/501100011033 and by ESF+}

\date{\today}

\maketitle

 \begin{center}
 \begin{minipage}{.8\textwidth}
 \section*{Abstract}
We are interested in the following conjecture of Wilson from 1981: every locally nilpotent countably
categorical group is nilpotent. Following our previous work on the Lie algebra analogue of the conjecture, we use Lie methods to deduce the first nontrivial cases of the Wilson conjecture for $n$-Engel groups: countably categorical $3$-Engel groups and $4$-Engel $5$-groups are nilpotent, from which we also conclude that countably categorical $4$-Engel groups of odd exponent are nilpotent. This is implemented via an exceptional case of the Lazard correspondence, checked using computer algebra systems. We also study the transfer of nilpotency results between the three categories: groups, Lie algebras, associative algebras. Among other things, we prove that the Wilson conjecture implies the analogous nilpotency statement for associative algebras (modulo the commutative case). 
 \end{minipage}
 \end{center}


\section{Introduction}

A mathematical structure $M$ (groups, rings, algebras...) is $\omega$-categorical if its first-order theory admits a unique countable model up to isomorphism. By the Ryll-Nardzewski Theorem, this is equivalent to the following more dynamical definition: $M$ is $\omega$-categorical if for each $n\in\N$, the componentwise action of $\Aut(M)$ on the Cartesian power $M^n$ has only finitely many orbits. The reader unfamiliar with model theory may take the latter as a definition. All commutators or brackets in this paper are left-normed ($[x,y,z] = [[x,y],z]$).

\subsection{Origin of the problem: classifying $\omega$-categorical associative rings} The study of $\omega$-categorical rings flourished in the second half of the 1970s. It begins with the work of Macintyre and Rosenstein \cite{MacintyreRosenstein1976}, where a classification of unital and reduced (i.e. without nilpotent elements) $\omega$-categorical rings is given, in terms of filtered Boolean powers of finite fields à la Arens-Kaplansky. In parallel, Baldwin and Rose \cite{baldwinrose1977} take the approach of studying on one side the Jacobson radical $J(R)$ of a given $\omega$-categorical ring $R$, and on the other side the semiprimitive quotient $R/J(R)$. While they obtain little information on $R/J(R)$ except in the case where $R$ is also model-theoretically stable, they prove that $J(R)$ is always a nilring and ask if it is nilpotent. This will be answered by Cherlin in two beautiful papers, one dealing with the commutative case and the second one with the general case. 

\begin{fact}(Cherlin Theorem I \cite{cherlinnilringI})
If $A$ is a commutative nilring and the number of orbits in the action of $\Aut(A)$ on $A\times A\times A$ is finite, then $A$ is nilpotent. In particular,  an $\omega$-categorical commutative nilring is nilpotent. 
\end{fact}

\begin{fact}(Cherlin Theorem II \cite{cherlinnilringII})
If $A$ is a nilring and the number of orbits in the action of $\Aut(A)$ on $A\times A\times A\times A\times A\times A$ is finite, then $A$ is nilpotent. In particular, an $\omega$-categorical nilring is nilpotent. 
\end{fact}

Note that for any $\omega$-categorical commutative ring $R$, the quotient $R/J(R)$ is reduced and hence falls into the analysis of Macintyre and Rosenstein, so that we get a complete picture. In the general case, a classification is far from complete for $\omega$-categorical rings without extra assumptions.

The results of Cherlin are to be read in parallel with the classical theorem of Nagata \cite{Nagata1952}, Higman \cite{Higman1956nagata}, Dubnov and Ivanov \cite{DubnovIvanov1943}:

\begin{fact}[Nagata, Higman, Dubnov, Ivanov]
    Every associative algebra of bounded nilexponent $n$ over a field of characteristic $0$ or $p>n$ is nilpotent.
\end{fact}

The nilpotency class of the ring depends only on $n$. The best known bound is $n^2$, obtained by Razmyslov \cite{Razmyslov1974TraceIdentities} and the exact bound is conjectured to be $\frac{n(n+1)}{2}$. This raises the question of the existence of such a bound in the case treated by Cherlin: does there exist a function $f$ such that an $\omega$-categorical nilring of nilexponent $n$ is nilpotent of class at most $f(n)$?

The work of Cherlin is also to be connected to the solution to the Kurosh-Levitzki for nilrings by Jacobson and Kaplansky \cite{Kaplansky1946, Jacobson1945}: nilrings of bounded nilexponent are locally nilpotent. In particular, finitely generated nilrings of bounded nilexponent are nilpotent. From a certain point of view, $\omega$-categoricity could be considered as a very weak form of being finitely generated, we will come back on this later on. Note that being a nilring and a locally nilpotent ring are equivalent notions for locally finite rings, in particular for $\omega$-categorical rings (see section \ref{sec:assalgWCGWCLA}).

\subsection{The Wilson conjecture on countably categorical groups}

 Recall that a group $G$ is $n$-Engel if it satisfies the identity 
\[[x,\underbrace{y,\ldots, y}_{n\text{ times}}] = 1.\]

In a foundational paper \cite{wilson1981}, Wilson starts a systematic study of $\omega$-categorical groups and states a conjecture which is easily seen to be equivalent to the following statement (see \cite{delbee3Engelchar5}).
\begin{conjecture}[WCG, Wilson, 1981]\label{conj:wilson}
    Every $\omega$-categorical $n$-Engel $p$-group is nilpotent.
\end{conjecture}

Under the WCG, in every $\omega$-categorical group $G$ there is a finite characteristic series $G = G_1\rteq G_2\rteq \ldots \rteq G_n = 1$ such that each quotient $G_i/G_{i+1}$ is either an elementary abelian group or a boolean power of a finite simple group, see \cite{Apps83A}. A positive answer to Wilson's conjecture would also imply that every locally nilpotent $\omega$-categorical group is nilpotent. 

The conjecture has been solved under extra model-theoretic assumptions but seems overall out of reach for now. A few general structural results were obtained in the 1980s by Rosenstein, Cherlin, Wilson, Apps \cite{Rosenstein1973, wilson1981, Apps83A,Apps83B, Apps83C} then later by Macpherson \cite{Macpherson1988ubiquitousomcatgroups} and Archer and Macpherson \cite{ArcherMacpherson1997}. 
Model theorists have been quite successful in analyzing the structure of \textit{tame} $\omega$-categorical groups, starting with Felgner \cite{Felgner1978} who described the stable ones, followed then by a series of results by Baur, Cherlin, Macintyre  \cite{BCM79}, Macpherson (e.g. Wilson's conjecture holds for NSOP groups \cite{Macpherson1988ubiquitousomcatgroups}), Krupinski, Wagner, Evans, Derakhshan, Dobrowolski \cite{DerakhshanWagner1997, EvansWagner2000, Krupinski2012, DobrowolskiWagner2020} (see \cite{DobrowolskiWagner2020} for a nice history section of the advances in this approach). Very recently Macpherson and Tent \cite{macpherson2024omegacategoricalpseudofinitegroups} studied pseudofinite $\omega$-categorical groups, with refined conjectures in that case.

The analogy between the WCG and Cherlin's results is apparent, and suggests a naive attack on the WCG by revisiting Cherlin's methods, in a group context, whatever that could mean. In fact, in \cite{wilson1981}, Wilson mentions: ``Powerful though Cherlin's methods in \cite{cherlinnilringI} and \cite{cherlinnilringII} are, they seem inadequate
for an attack on [the WCG].". Our conviction is that, although at a global level, the methods and heuristics of Cherlin might not indeed be adequate in the group context, they might prove to be useful at a local level, i.e. for particular values of parameters. We will obtain the main result of this paper, namely the WCG for $4$-Engel $5$-groups, using a direct translation of Cherlin's proofs, requiring an intermediate category sitting between groups and associative algebras: Lie algebras.

\subsection{An intermediate ``test question"} In \cite{Apps83A}, Apps gives the formal arguments leading to the analysis of $\omega$-categorical groups via a characteristic series as mentioned above. The author also writes ``locally nilpotent $\aleph_0$-categorical Lie rings might serve as an interesting intermediate case". The analogous of the WCG for Lie rings or rather Lie algebras is the following statement, which we will abusively refer to as the Wilson conjecture for Lie algebras (WCLA), even though it was not stated by Wilson. Recall that a Lie algebra $L$ is $n$-Engel if it satisfies the identity 
\[[x,\underbrace{y,\ldots, y}_{n\text{ times}}] = 0.\]

\begin{conjecture}[WCLA]
 Every $\omega$-categorical $n$-Engel Lie $\F_p$-algebra is nilpotent.
\end{conjecture}

The fact that the Lie ring case might actually be an intermediate case, as suggested by Apps, is not so clear. In fact, section \ref{sec:assocliering} deals with this particular question and our conclusions seem to indicate that, while the WCG and the WCLA are similar statement, they might just be unrelated questions, at least at a global level. However, as mentioned above, at a local level, by which we mean for some particular values of parameters $n,p$, transfer between groups, Lie algebras and associative algebras can be used to leverage results from one category to another.

This paper is the third in a series of the author on the WCG and the WCLA. Our first paper \cite{delbee3Engelchar5} explains our overall strategy. Starting from the following scheme
\begin{figure}[h!]
    \centering
\includegraphics[scale=.5]{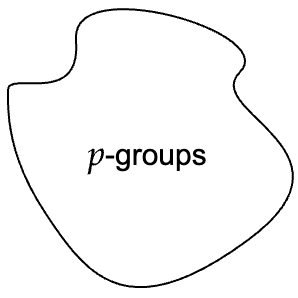} \quad \raisebox{3em}{\scalebox{1.5}{$\leadsto$}}  \quad  
\includegraphics[scale=.5]{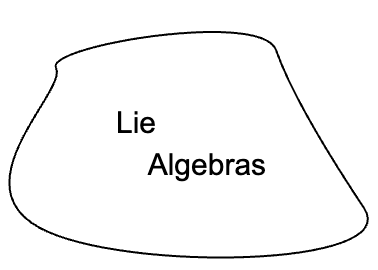} \quad \raisebox{3em}{\scalebox{1.5}{$\leadsto$}}  \quad  \includegraphics[scale=.5]{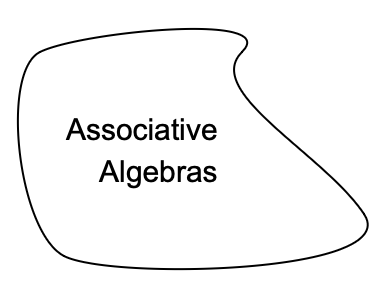}    
\end{figure}
our the long-term goal is to be able to trace nilpotency results back from the right of an arrow to the left of the arrow.  This could be at a heuristical level, for instance, in \cite{delbee3Engelchar5} we solve the WCLA for $3$-Engel Lie $\F_5$-algebras by a direct implementation of Cherlin's argument in the enveloping algebra of $3$-Engel Lie algebras of characteristic $5$. In \cite{delbee4engelchar3}, we interpret an associative algebra in $4$-Engel Lie $\F_3$-algebras, applying directly the result of Cherlin. In the present paper, we go one step further into our program and deduce our first cases of the WCG from cases of the WCLA, namely we will transfer the result of $3$-Engel Lie $\F_5$-algebras to a result on $3$-Engel $5$-groups and then use that result to deduce the $4$-Engel $5$-groups case of the WCG. This is, to our knowledge, a first and highly non-trivial general case of the WCG.

 We also provide a few more results concerning the transfer of nilpotency from the left to the right direction. In section \ref{sec:lefttoright}, we deduce Cherlin Theorem II from the WCG (or the WCLA) assuming only Cherlin Theorem I (the commutative case). The proof is surprisingly nontrivial, see Theorem \ref{thm:WCLAimpliesCherlinII}. We further prove that an $n$-Engel Lie algebra which is the associated Lie algebra of an $\omega$-categorical associative algebra is nilpotent, see Theorem \ref{thm:shalev}. The latter is an analogue of a classical theorem of Shalev \cite{shalevlocalengelnilp}, and we also provide a group version, see Corollary \ref{cor:shalevgroup}.

\subsection{Analogies with classical nilpotency problems} In \cite{delbee3Engelchar5, delbee4engelchar3}, we draw strong connections between the WCG, the WCLA and Cherlin Theorem II on one side, and the following  classical nilpotency problems 
\begin{itemize}
\item The \textit{restricted Burnside problem}: there are only finitely many finite $r$-generated groups of bounded exponent $e$, for each given pair $(r,e)$.
\item The \textit{Engel group problem}: every $n$-Engel group is locally nilpotent.
\item The \textit{Kostrikin problem}: every $n$-Engel Lie algebra is locally nilpotent.
\item The \textit{Kurosh-Levitzki problem}: every nilalgebra of bounded nilexponent $n$ is locally nilpotent. 
\end{itemize}
The Kostrikin problem and the restricted Burnside problem were solved by Zelmanov \cite{Zelmanov1990,Zelmanov1991}. The Kurosh-Levitzki problem was positively solved by Jacobson and Kaplansky \cite{Kaplansky1946, Jacobson1945}. The Engel group problem is still open, although it seems that specialists believe it to be too good to be true.

We also draw connections with \textit{global} nilpotency results, which are often obtained via asymptotic assumptions such as the following classical ``characteristic $0$" theorem of Zelmanov \cite{Zelmanov88}: for any $n$, there exists $N\in \N$ such that every $n$-Engel Lie algebra of characteristic $p>N$ is nilpotent. The WCLA and the WCG can be considered as nilpotency problems sitting in between local nilpotency problems (e.g. the Kostrikin problem) and global nilpotency problems (e.g. the characteristic $0$ theorem). Indeed, statements of the WCG and the WCLA are global nilpotency problems, but the $\omega$-categoricity assumption is stronger than just a local finiteness assumption and in some sense, could be considered close to a very weak form of being finitely generated. In fact, another conjecture of Wilson states that $\omega$-categorical groups generate a \textit{Cross} variety, i.e. a variety generated by a finite group.

The connections between the WCG, the WCLA and those classical problems are not only anecdotal and aesthetically satisfying. Most importantly, they come into play at various levels for concrete new results on the WCG and the WCLA. For instance, the author's solution of the WCLA for $n = 4$ and $p = 3$ used machinery directly ocurring in the work of Higman, Kostrikin, Vaughan-Lee and Traustason, revisited in the $\omega$-categorical context. This work lies at the intersection of algebra, computer algebra and model theory. Another successful approach developed by the classical authors for solutions of the problems above involves realizing the strategy that we seek: passing from one category to another, between groups, Lie and associative algebras. 

\subsection{Lie methods in group theory}
There is a remarkable number of papers which advertise \textit{Lie methods in group theory} in their title \cite{VaughanLee2011LiemethodsEngelgroups, Zelmanov95Lieringmethods, liemethodswall, BartholdiGrigorchuk2000LieMethods,Shumyatsky2000LieRingMethods, Shalev2000LieMethods}. It is true that, on paper, this approach is very attractive. It roughly consists in solving a group-theoretic problem via the following steps:
\begin{enumerate}
\item construct a Lie algebra from a group,
\item transfer the group-theoretic problem into a Lie algebra problem,
\item solve the Lie algebra problem.
\end{enumerate}
For instance, the restricted Burnside problem follows from the Kostrikin problem above, and the latter is what Zelmanov had to solve \cite{Zelmanov1990,Zelmanov1991}. The general philosophy seems to be that the inherent linearity of Lie algebras sometimes gives the right point of view in order to solve a given problem. The usual approach of Lie methods in group theory uses the \textit{associated Lie ring} of a given group $G$, which is obtained by defining a Lie bracket (in terms of the group commutators) on the external direct sum $\oplus_{i<\omega}\gamma_i/\gamma_{i+1}$ where $(\gamma_{i})_i$ is the lower central series. This construction is expounded in section \ref{sec:assocliering} and our conclusions are that this approach is sure to fail for solving the WCG. Indeed, the associated Lie ring of any $\omega$-categorical group is nilpotent, and seems to carry very little information concerning the nilpotency of the group. 

Nonetheless, we succeeded in transferring the nilpotency result of \cite{delbee3Engelchar5} to the group context by the use of an exceptional case of the \textit{Lazard correspondence}, which we describe at length in section \ref{subsec:lazard}. Our result is that $3$-Engel groups of exponent $5$ and $3$-Engel Lie $\F_5$-algebras enjoy an equivalence of categories, a seemingly rare phenomenon which seems to have been missed by the specialists of the domain. This purely algebraic result also has group-theoretic consequences, such as an analogue, for $3$-Engel groups, of a classical theorem of Higgins on Lie algebras (Corollary \ref{cor:higginsforgroups}): solvability and nilpotency coincide. We also draw similar conclusions for $4$-Engel groups of exponent $7$.

\subsection{Current state of knowledge for small values of $(n,p)$} We summarize the results concerning the WCG and the WCLA in the following table. The intersection of an ``$n$-Engel" row and a ``$p$" column is -if it exists- a bound on the nilpotency class of $n$-Engel $p$-groups, or $n$-Engel Lie $\F_p$-algebras.

\begin{table}[h]
\centering
\begin{minipage}{0.45\linewidth}
\centering
\begin{subtable}{\linewidth}
\centering
\begin{tblr}{
 cells={c},
 vline{1-2,7} = {-}{},
 hline{1-2,6} = {-}{}
}
$p=$ & $2$ & $3$ & $5$ & $7$ & $>7$ \\
$2$-Engel & 2 & $3$ & $2$ & $2$ & $2$ \\
$3$-Engel & \circleblue{$\infty$}$\mathrlap{^\ast}$ & $4$ & \hspace{-.5em} \tikzmark{left} \circleblue{$\infty$} & $4$ & $4$ \\
$4$-Engel & $\infty$ & \circleblue{$\infty$}$\mathrlap{^\ast}$ & \circleblue{$\infty$} & $7$ & $7$ \\
$5$-Engel & $\infty$ & $\infty$ & $\infty$ & $\infty$ & $10$
\end{tblr}
\caption{Engel $p$-groups}
\end{subtable}
\end{minipage}
\hfill
\begin{minipage}{0.45\linewidth}
\centering
\begin{subtable}{\linewidth}
\centering
\begin{tblr}{
 cells={c},
vline{1-2,7} = {-}{},
 hline{1-2,6} = {-}{}
}
$p=$ & $2$ & $3$ & $5$ & $7$ & $>7$ \\
$2$-Engel & $2$ & $3$ & $2$ & $2$ & $2$ \\
$3$-Engel & $\infty$ & $4$ & 
\hspace{-.5em} \tikzmark{right} \circleblue{$\infty$} & $4$ & $4$ \\
$4$-Engel & $\infty$ & \circleblue{$\infty$} & $\infty$ & $7$ & $7$ \\
$5$-Engel & $\infty$ & $\infty$ & $\infty$ & $\infty$ & $11$
\end{tblr}
\caption{Engel $\mathbb{F}_p$-LAs}
\end{subtable}
\end{minipage}

\begin{tikzpicture}[overlay,remember picture]
\draw[<->, crimsonred, bend left=6] ($(left)+(.5,.3)$) to[bend left] node[pos=0.40,above] {\footnotesize Exceptional Lazard} ($(right)+(.1,.25)$);
\end{tikzpicture}
\end{table}

A blue circle indicates that the WCG or the WCLA has been proven positively by the author for that case. A star indicates that the nilpotency in that case follows directly from the literature: those cases are solvable in general and hence $\omega$-categorical ones are nilpotent by a Theorem of Wilson \cite{wilson1981}.

\subsection{From algebras to rings} Our theorems generally concern associative or Lie algebras over finite fields, which might seem restrictive at first sight. First, the reason to consider positive characteristic is because all nilpotency problems are trivial in characteristic $0$, by classical theorems such as Nagata-Higman \cite{Nagata1952,Higman1956nagata} or Zelmanov \cite{Zelmanov88}. It remains the question of associative or Lie rings, and our results extend from algebras to rings in a quite direct way. This is based on the following two facts, for an $\omega$-categorical (Lie/associative) ring $R$
\begin{itemize}
    \item the group $(R,+)$ has bounded exponent $n = p_1^{\alpha_1}\cdots p_s^{\alpha_s}$,
    \item for each $i$ the quotient $R/p_iR $ is interpretable, hence again $\omega$-categorical.
\end{itemize}
Then, assuming that a nilpotency result has been proved for a class of $\omega$-categorical algebras, it generally extends to the $\omega$-categorical ring case. Let us take the example of Theorem \ref{thm:shalev} below, which states that if $A$ is an $\omega$-categorical associative algebra whose underlying Lie algebra $L(A)$ (with bracket $[x,y] = xy-yx$) is $n$-Engel, then $L(A)$ is nilpotent as a Lie algebra. Let $R$ be an $\omega$-categorical ring of characteristic $e = p_1^{\alpha_1}\cdots p_s^{\alpha_s}$ such that the Lie ring $L(R)$ is $n$-Engel for some $n$. For each $1\leq i\leq s$, the $\F_{p_i}$-algebra $R/p_iR$ is $\omega$-categorical and $L(R/p_iR)$ is again $n$-Engel. Applying Theorem \ref{thm:shalev}, there exists $e_i$ such that $L(R/p_iR)$ is nilpotent of class $<e_i$. Therefore, the bracket of $e_i$ elements of $R$ falls in $p_iR$ and the bracket of $\alpha_i e_i$ elements of $R$ falls in $p_i^{\alpha_i}R$. In the end, the product of $\alpha_1 e_1+\ldots+\alpha_s e_s$ elements falls in $eR = 0$, hence $L(R)$ is nilpotent.

All of the statements from Section \ref{sec:lefttoright} extend in a similar fashion from algebras to rings, with the appropriate adjustment. The only thing to check is that the hypotheses of the theorem are preserved under the map $R\to R/pR$.

\subsection*{AI Disclosure} I used chatGPT to proofread the paper, correct mistakes and inaccuracies. None of the proofs have had any mathematical input from AI. 


\subsection*{Acknowledgement} 
As usual, I have a long list of colleagues and friends to thank. First, I am very thankful to Michael Vaughan-Lee for numerous discussions and exchanges, and for warm welcomes in the lowlands of Scotland. I am also very grateful to Eamonn O'Brien for his patience and help with several NQA runs which I was not (yet) able to exploit. I am very grateful to Dugald Macpherson for numerous insights and for his readiness to discuss. I am very thankful to Nick Ramsey, for our endless discussions and for your very valuable friendship. I am also grateful to Willem de Graaf for helpful exchanges. I am thankful to the various places that welcomed me during the writing of this note, mainly a woody Inn of the Appalachians in West Virginia, and Turkish Airline. Azkenik, ene maitiari, zure patzientizagatik, zure goxotasunagatik eta elkarren ondoan beti izateagatik, mundua pasatzen begiratuz.

\section{The rightward direction: the natural reductions}\label{sec:lefttoright}

\subsection{From groups and Lie algebras to associative algebras}\label{sec:assalgWCGWCLA} In this section, we prove that Cherlin Theorem II (or rather its qualitative version) follows from the WCG and the WCLA, using Cherlin Theorem I, the commutative case, as the main tool. One should not expect to prove anything concerning a commutative associative algebrausing results from Lie algebras.

A nilalgebra or a nilring always has a Lie algebra reduct and a group reduct. Those will be our way to connect the WCG and the WCLA to Cherlin's theorem.

\noindent $\bullet$ The \textit{associated Lie algebra $L(A)$ of an associative algebra} $A = (A,+,.)$ is given by the structure $(A,+,[.,.])$ where $[x,y] = xy-yx$. By denoting $\gamma_i(L(A))$ the member of the (Lie-algebraic) lower central series, the following is clear:
\[\gamma_k(L(A)) \seq A^{(k)}.\]
Therefore, if $A$ is nilpotent as an associative algebra, so is $L(A)$ as Lie algebra.

\noindent $\bullet$ The \textit{adjoint group $A^\circ$ of a nilalgebra} $A$ is given by the structure $(A, \circ)$, where $x\circ y = x+y+xy$. The adjoint group corresponds to the group of units of the \textit{unitarization} $\F_p+A$ of $A$, obtained by adding a unit. The arithmetic in $\F_p+A$ is given by:
\[(\lambda+x)(\mu+y) = \lambda\mu+\mu x+\lambda y+xy\]
The set $1+A$ is a multiplicative subset of $\F_p+A$, and one immediately sees that the arithmetic within $1+A$ coincides with the operation $\circ$:
\[(1+x)(1+y) = 1+x+y+xy\]
The inverse for $\circ$ of an element $a\in A$ is given by $a^{-1} = \sum_{i\geq 1} (-a)^i $, which is well-defined in a nilring. In $1+A$, we have $(1+a)^{-1} = \sum_{i \geq 0} (-a)^{i}$. Let $G$ be the group $1+A$ and $g,h\in G$, $g = 1+x$, $h = 1+y$. We will denote the group commutator in $G$ by $[g,h]_G = g^{-1}h^{-1}gh$. Commutation in the group can be controlled by the Lie commutator of $A$:
\[gh = hg + [x,y]\]
Therefore, commutators are related as follows:
\[[g,h]_G = 1+g^{-1}h^{-1}[x,y] = 1+[x,y]+z[x,y]\tag{$\star$}\]
by writing $g^{-1}h^{-1} = 1+z$, $z\in A$.
It follows that $\gamma_2(G)\seq 1+A^{(2)}$
and an immediate induction using $(\star)$ yields:
\[\gamma_k(G)\seq 1+A^{(k)}.\]
Therefore, if $A$ is nilpotent as an associative algebra, so is $A^\circ$ as a group. 

There is a rather rich literature on transfers of properties between nilrings (or more generally radical rings), the associated Lie algebra and the adjoint group, see for instance \cite{jenningsnilpradicalrings,AmbergDickenschied1995, AmbergSysak2002, shalevlocalengelnilp}.

We now quickly argue that the following notions coincide for $\omega$-categorical rings:
\begin{enumerate}
    \item locally nilpotent,
    \item uniformly locally nilpotent,
    \item nil,
    \item nil of bounded nilexponent.
\end{enumerate}

By the Ryll-Nardzewski Theorem, the increasing union of sets $\set{x\mid x^n = 0}$ stabilizes which immediately yields $(3)\implies (4)$, hence $(3)\Leftrightarrow (4)$. A similar argument yields $(1)\implies (2)$ hence $(1)\Leftrightarrow (2)$.
$(1)\implies (3)$ is immediate and $(4) \implies (1)$ \textit{a priori} uses the solution to the Kurosh-Levitzki problem:

\begin{fact}[Kurosh-Levitzki Problem \cite{Jacobson1945,Kaplanski1948}]
    Any finitely generated nilring of bounded nilexponent is nilpotent.
\end{fact}
Using that $\omega$-categorical structures are uniformly locally finite, we will only need the following baby version of the Kurosh-Levitzki problem.

\begin{fact}\label{fact:kurochGPT}
    Any finite nilring is nilpotent.
\end{fact}

In the proof of the next result, we will use the left-normed convention for writing functions. In particular, $(x)f\circ g$ is the result of applying first $f$ to $x$ and $g$ to the image $(x)f$ of $x$ by $f$. This is convenient for the adjoint maps $x\mapsto [x,b]$ for $b$ in a Lie algebra $L$, and is more suitable with our left-normed convention of bracketting. We are ready to prove our first result.

\begin{theorem}\label{thm:WCLAimpliesCherlinII}
    Assuming Cherlin Theorem I, any $\omega$-categorical nilalgebra whose associated Lie algebra is nilpotent as a Lie algebra is nilpotent as an associative algebra. In particular, assuming Cherlin Theorem I, the WCLA implies Cherlin Theorem II.
\end{theorem}
\begin{proof}
    Let $(A,+,\cdot)$ be an $\omega$-categorical associative nilalgebra and assume that the associated Lie algebra $ (L(A), +, [.,.])$ is nilpotent. There exists $k\in \N$ such that $\gamma_{k+1}(L) = 0$. In particular, we have, for all $a, b_1,\ldots, b_k\in A$
    \[a[b_1,\ldots,b_{k+1}] = 0.\]
    We prove the result by induction on $k$ satisfying the latter equation.

    \textit{The case $k = 1$.} Assuming that $A$ satisfies the equation $x[y_1,y_2] = 0 $, we have $ab_1b_2 = ab_2b_1$ for all $a,b_1,b_2\in A$ and thus the associative subalgebra $B$ of $\End(A)$ generated by the linear maps $f_b:x\to xb$ is commutative. Fixing $a\in A$, we consider the set $B_a$ of germs of $B$ at $a$. Those are defined by the quotient of the algebra $B$ by the equivalence relation $f_b\sim_a f_c\iff (a)f_b = (a)f_c$. Addition of germs is given by taking the class of $f_b+f_c$. Because $(a)f_bf_c = (a)f_cf_b$, the multiplication in $B$ yields a well-defined multiplication on germs, and $B_a$ is a commutative associative algebra. 
    
    
    We claim that $B_a$ is interpretable in $A$, using the parameter $a$. The proof of this fact is directly analogous to the one of \cite[Theorem 3.1.]{delbee4engelchar3}, $\omega$-categoricity yields $m\in \N$ such that every element of $B_a$ is equivalent to a sum of the form $\sum_{i = 1}^m f_{b_i}$ for a uniform $m$ (note that $f_{c_1}\circ \cdots \circ f_{c_r} = f_{c_1\cdots c_r}$). Then $B_a$ is interpreted as a quotient of the cartesian power $A^m$. In particular, $B_a$ is $\omega$-categorical.

    Because $A$ is uniformly locally nilpotent, $B_a$ is a nilalgebra and Cherlin Theorem I applies so $B_a$ is nilpotent. It follows that there exists $s\in \N$ such that $f_{b_1}\cdots f_{b_s} \sim_a 0$. In other words $ab_1\cdots b_s = 0$. A priori, $s$ depends on $a$ but by $\omega$-categoricity, there exists a uniform $t\in \N$ such that for all $a,b_1,\ldots, b_t$ we have $ab_1\cdots b_t = 0$ and hence $A$ is nilpotent.

    \textit{The case $k> 1$.} Assume that $k>1$. For each $a,b_1,\ldots,b_{k+1} \in A$, we have $a[[b_1,\ldots, b_k],b_{k+1}] = 0$, therefore
    \[a[b_1,\ldots, b_k]b_{k+1} = ab_{k+1}[b_1,\ldots, b_k] \]
    \begin{claim}
        There exists $l\in \N$ such that for all $a\in A$ and $(b_{i,j})_{1\leq i\leq l,1\leq j\leq k}\in A$ we have
        \[a[b_{1,1},\ldots,b_{1,k}][b_{2,1},\ldots,b_{2,k}]\cdots [b_{l,1},\ldots,b_{l,k}] = 0\tag{$\dagger$}\]
    \end{claim}
    \begin{proof}[Proof of the Claim]
    We use the same trick as in the case of $k = 1$: we interpret germs of a subalgebra of $\End(A)$. 
    
    The equation $x[y_1,\ldots, y_k]z = xz[y_1,\ldots, y_k]$ implies in particular that 
    \[a[b_1,\ldots, b_k] [c_1,\ldots,c_k] = a[c_1,\ldots, c_k][b_1,\ldots, b_k].\]
    We consider the subalgebra $C$ of $\End(A)$ generated by linear maps of the form $f_{\vec b}:x\mapsto x[b_1,\ldots,b_k]$, for $\vec b = (b_1,\ldots, b_k)$ in the cartesian power $A^k$. We consider the germs $C_a$ of elements of $C$ at $a$, and the commutation above gives that $C_a$ is a well-defined, commutative associative algebra. It is again interpretable in $A$ although here we need to observe that $f_{\vec b}f_{\vec c}$ is not necessarily of the form $f_{\vec d}$. Nonetheless, $\omega$-categoricity yields that there exists $m$ such that every element in $C_a$ is equivalent to an element of the form 
    \[\sum_{i = 1}^m f_{\vec b_1}\cdots f_{\vec b_m}\]
    so that $C_a$ can be interpreted as a quotient of $\bigcup_{1\leq j\leq m} (A^{m})^j$. We conclude that $C_a$ is $\omega$-categorical.

    Again, as $A$ is uniformly locally nilpotent, the algebra $C_a$ is nil and hence nilpotent by Cherlin Theorem I. This implies that for each $a$, there is an $l\in \N$ as in the statement of the claim. Then, $\omega$-categoricity yields a uniform $l\in \N$.
    \end{proof}

    Consider now the two-sided associative ideal $I$ of $A$ generated by all elements of the form $a[b_1,\ldots,b_k]$ for all $a,b_i\in A$. Note that $I$ is not necessarily commutative. Nonetheless, the latter implies that every element of $I$ can be written of the form $\sum_i a_i [b_{i,1},\ldots, b_{i,k}]$, and that the product of $l$ elements of $I$ can be written as a sum of products of the form 
    \[a_1\cdots a_l[b_{1,1},\ldots,b_{1,k}][b_{2,1},\ldots,b_{2,k}]\cdots [b_{l,1},\ldots,b_{l,k}]\]
    which vanishes by the Claim. We conclude that $I$ is a nilpotent ideal of $A$. The quotient algebra $A/I$ is again $\omega$-categorical and satisfies the equation $x[y_1,\ldots, y_k] = 0$, so it is nilpotent by induction. We conclude that $A$ is nilpotent.

    To conclude the ``in particular" statement, let $(A,+,\cdot)$ be an $\omega$-categorical associative nilalgebra, and consider the associated Lie algebra $ (L(A), +, [.,.])$. By the beginning of the section, the algebra $A$ is uniformly locally nilpotent. Let $a,b\in L(A)$ and $A_0 = \vect{a,b}$ be the associative algebra generated by $a,b$. Then for each $n$, the Lie element $[a,b^{(n)}]$ lives in $\gamma_n(L(A_0))\seq A_0^{(n)}$ hence there exists $n\in \N$ such that $L := L(A)$ is $n$-Engel. By the WCLA, $L$ is nilpotent, hence we conclude by the previous statement.
\end{proof}

\begin{remark}
Note that the proof does not use that $L$ is nilpotent, but rather that $A\gamma_k(L(A)) = 0$ for some $k$.
\end{remark}

The fact that the WCG implies Cherlin Theorem II will rely on Theorem \ref{thm:WCLAimpliesCherlinII} via the well-known following fact due to Jennings \cite{jenningsnilpradicalrings}:
\begin{fact}\label{fact:jennings}
    Let $A$ be a nilalgebra. Then $A^\circ$ is nilpotent if and only if $L(A)$ is nilpotent.
\end{fact}

\begin{corollary}
    Assuming Cherlin Theorem I, any $\omega$-categorical nilalgebra whose adjoint group is nilpotent is nilpotent as an associative algebra. In particular, assuming Cherlin Theorem I, the WCG implies Cherlin Theorem II.
\end{corollary}
\begin{proof}
    Let $A$ be an $\omega$-categorical nilalgebra whose adjoint group $A^\circ$ is nilpotent. Then by Fact \ref{fact:jennings}, $L(A)$ is nilpotent and hence by Theorem \ref{thm:WCLAimpliesCherlinII}, $A$ is nilpotent. To conclude the ``in particular" statement: assume that $A$ is an $\omega$-categorical nilalgebra. As $A$ is locally nilpotent as an associative algebra and the $\circ$-group generated by $a,b$ is also nilpotent, we have that $A^\circ$ is $n$-Engel for some $n\in \N$. Applying the WCG, we have that $A^\circ$ is nilpotent and we conclude.
\end{proof}

We conclude this section by a quick proof of very special cases of the WCG and the WCLA, where the group or the Lie algebra sits within an $\omega$-categorical associative algebra. The Lie algebra case is an analogue of the following classical Theorem of Shalev \cite{shalevlocalengelnilp}.

\begin{fact}[Shalev]
    Let $A$ be a finitely generated associative algebra over a field of positive characteristic. If $L(A)$ is $n$-Engel then $L(A)$ is nilpotent. 
\end{fact}

\begin{theorem}\label{thm:shalev}
    Let $A$ be an $\omega$-categorical associative algebra. If $L(A)$ is $n$-Engel then $L(A)$ is nilpotent as a Lie algebra.
\end{theorem}

\begin{proof}
    Let $A$ be an $\omega$-categorical associative algebra such that $L(A)$ is $n$-Engel. Assume that $L(A)$ is over $\F_p$. Let $I\seq A$ be the two-sided ideal (in the sense of associative algebras) generated by $\gamma_2(L(A))$. Note that $I$ is $\Aut(A)$-invariant.
    
    We claim that $I$ is nil. Let $u\in I$, say $u = c_1[a_1,b_1]d_1+\ldots +c_s[a_s,b_s]d_s$. Let $A_0$ be the subalgebra generated by $(a_i,b_i,c_i,d_i)_i$ and let $J\seq A_0$ be the Jacobson radical of $A_0$. $A_0/J$ is semisimple Artinian, hence by Artin-Wedderburn it is a product of matrix rings over finite fields. As $L(A_0)$ is a finite $n$-Engel Lie algebra, it is nilpotent by Engel's theorem and so is $L(A_0/J)$. It follows that $A_0/J$ is a product of fields, in particular it is commutative. Therefore the ideal $I_0$ of $A_0$ generated by $\gamma_2(L(A_0))$ is contained in $J$, which is nilpotent. It follows that $u = c_1[a_1,b_1]d_1+\ldots +c_s[a_s,b_s]d_s\in I_0$ satisfies $u^n = 0$ for some $n\in \N$.

    Applying Cherlin, we have that $I$ is nilpotent, say $I^{(k)} = 0$. We have $\gamma_2(L(A))\seq I$, hence $[\gamma_2(L(A)),\gamma_2(L(A))]\seq I^{(2)}$ and inductively, we conclude that the Lie algebra $L(A)$ is solvable. By \cite[Corollary 3.3]{delbee4engelchar3}, an $\omega$-categorical $n$-Engel solvable Lie algebra is nilpotent, hence we conclude.
\end{proof}

Our last result uses the following fact of Amberg and Sysak \cite{AmbergSysak2000}.

\begin{fact}\label{fact:ambergsysakengel}
    Let $A$ be a nilalgebra. Then $A^\circ$ is $n$-Engel if and only if $L(A)$ is $m$-Engel.
\end{fact}

Note that $n$ and $m$ may effectively differ, see \cite{DeryabinaKrasilnikov2019}.

\begin{corollary}\label{cor:shalevgroup}
    Let $A$ be an $\omega$-categorical associative  nilalgebra. If $A^\circ$ is $n$-Engel then $A^\circ$ is nilpotent.
\end{corollary}

\begin{proof}
    By Fact \ref{fact:ambergsysakengel}, $L(A)$ is $m$-Engel for some $m$ hence by Theorem \ref{thm:shalev}, $L(A)$ is nilpotent and hence by Fact \ref{fact:jennings} $A^\circ$ is nilpotent.
\end{proof}



\subsection{From groups to Lie algebras: limitations and open questions}\label{sec:assocliering} First, let us recall basics on what is called \textit{Lie methods in group theory}. This approach was successful in the solution to the restricted Burnside problem \cite{Zelmanov1990,Zelmanov1991}, see also  \cite{Zelmanovfieldslecture}, was extended to various settings, for instance in the study of $n$-Engel groups \cite{Traustason2011EngelGroups}. Let us see how one can generally construct a Lie algebra or a Lie ring from any group $G$. Starting from the lower central series $(\gamma_i)_{i\geq 1}$, each successive quotient $\gamma_i/\gamma_{i+1}$ is an abelian group and one may consider the abelian group given by taking the direct sum of those:
\[L(G) = \bigoplus_{i\geq 1} \gamma_i/\gamma_{i+1}\]
Using identities for the group commutator (where $u^*$ means a conjugate of $u$), such as
\[[x, zy] = [x,y][x,z]^* \quad [xz,y] = [x,y]^* [z,y]\]
or the Hall-Witt identity
\[
 [x,y,z^*]^* [z,x,y^*]^* [y,z,x^*]^* = 1
\]
one checks that the group commutator $[.,.]: \gamma_i/\gamma_{i+1} \times \gamma_j/\gamma_{j+1}\to \gamma_{i+j}/\gamma_{i+j+1}$ may be extended to all elements of $L(G)$ and that the structure on $(L(G),+,[.,.])$ is that of a Lie ring. This ring is always graded by the sequence $(\gamma_i/\gamma_{i+1})_i$. The exponent of $G$ becomes the characteristic of $L(G)$. In particular, if $G$ has prime exponent $p$, then $L(G)$ is a Lie algebra over $\F_p$. In general, this construction is far from being meaningful as $L(G)$ might be trivial (for instance if $G$ is simple), but there are few cases where $L(G)$ detects a property of $G$. For instance, if $G$ is a finitely generated $p$-group of bounded exponent, then the LCS of $G$ is stationary if and only if $L(G)$ is nilpotent. The solution to the restricted Burnside problem is based on this fact.

Note that if one does not want to deal with Lie rings which are not algebras, they may consider the \textit{Zassenhaus series} $(d_n)_{n\geq 1}$, defined by the formula $d_{n} = \prod_{jp^k\geq n} \gamma_j^{p^k}$. Then the direct sum of the successive quotients
\[\bigoplus_{n\geq 1} d_n/d_{n+1}\]
always have the structure of a Lie algebra over $\F_p$.

Let us now specify this construction to the case of an $\omega$-categorical $p$-group $G$. By $\omega$-categoricity, the group $G$ has only finitely many characteristic subgroups, and all of those are definable. In particular, the series $(\gamma_i)_{i\geq 1}$ is stationary, that is, there exists $i_0$ such that $\gamma_j = \gamma_{i_0}$ for all $j\geq i_0$. It follows easily that the $i_0$-th member of the LCS of $L(G)$ vanishes, $L(G)$ is a nilpotent Lie ring. Note that the Lie structure on $L(G)$ is $0$-definable on the imaginary sorts $\gamma_1/\gamma_2\times \ldots \times \gamma_{i_0-1}/\gamma_{i_0}$. We conclude:
\begin{fact}
    Let $G$ be an $\omega$-categorical group, then the associated Lie ring $L(G)$ is interpretable, $\omega$-categorical and nilpotent.
\end{fact}

Of course, nothing depends on the particular filtration $(\gamma_i)_{i\geq 1}$ and the same statement will follow as soon as the associated Lie ring is constructed out of a characteristic series.

This should be considered as an admission of failure for the strategy proposed in the title of this article and any reasonable reader should already think of getting back to their business. Nonetheless, there are other ways of defining a Lie algebra from a group, as we will see in the next section.

The previous fact at least has the good taste of suggesting that the structure of an $\omega$-categorical group might be controlled by $\omega$-categorical nilpotent Lie rings and that a classification of those (which might be hopeless) would certainly be helpful for cracking the structure of $\omega$-categorical groups.

More importantly, the fact above explains the absence of any obvious connections between the WCG and the WCLA, as one could have hoped to deduce the WCG from the WCLA using the associated Lie ring. In fact, it raises the question of any connection at all between the WCLA and the WCG.
\begin{question}~
\begin{itemize}
    \item Does the WCG follow from the WCLA?
    \item Does the WCLA follow from the WCG?
    \item Are the WCG and the WCLA two independent statements?
\end{itemize}
\end{question}
Nonetheless, there are some cases where the WCG and the WCLA are related, for instance if they are in Lazard correspondence, as we will see in Section \ref{sec:lazardand4-Engel}.

Before ending this section, let us mention a last question, which concerns particular examples of $\omega$-categorical groups and Lie algebras. Given a fixed $c\in \N$ and prime $p>c$, the model-companion of the theory of nilpotent groups of class $c$ and exponent $p$ exists and is $\omega$-categorical, let us denote $\mathbf{G}_{c,p}$ the unique countable model. Similarly, the model-companion of the theory of nilpotent Lie algebras of class $c$ over $\F_p$ exists and is $\omega$-categorical, let us denote $\mathbf{L}_{c,p}$ the unique countable model. Those were studied recently in \cite{DINCI}, building on the work of Baudisch \cite{Baudisch2,Baudisch3} and Maier \cite{Maierexpp}. The two structures $\mathbf{G}_{c,p}$ and $\mathbf{L}_{c,p}$ are in Lazard correspondence (see Subsection \ref{subsec:lazard}), and are thus interdefinable. In general, starting from a group $G$ which is in Lazard correspondence with a Lie algebra $M$, the associated Lie ring $L(G)$ and $M$ might be completely different. Nonetheless, the genericity of $\mathbf{G}_{c,p}$ might circumvent this phenomenon, and we conjecture the answer to the following question to be positive.
\begin{question}
    Is $L(\mathbf{G}_{c,p})$ definably isomorphic to $\mathbf{L}_{c,p}$?
\end{question}
A positive answer to this question would imply the elimination of the imaginary sort $L(\mathbf{G}_{c,p})$, but we do not expect that the quotients $\mathbf{G}_{c,p}/\gamma_i$ are eliminable. We do expect that the imaginaries of $\mathbf{G}_{c,p}$ are eliminable up to those quotients.

\section{The leftward direction: the Wilson conjecture for $3$-Engel groups and $4$-Engel $5$-groups}\label{sec:lazardand4-Engel}

\subsection{Small values of $n$ and $p$} The goal of this section is the study of $\omega$-categorical $n$-Engel $p$-groups for small values of $n$ and $p$. Let us first recall some general facts about those notions.

Recall that a group $G$ is $n$-Engel if it satisfies the identity 
\[[x,\underbrace{y,\ldots, y}_{n\text{ times}}] = 1\]

The celebrated work of Zelmanov \cite{Zelmanov1990,Zelmanov1991} on the restricted Burnside problem yields that $n$-Engel Lie algebras are locally nilpotent. The corresponding statement for $n$-Engel groups is very much open.
\begin{question}
    Are $n$-Engel groups locally nilpotent?
\end{question}
The question seems to have been around for a long time, tracing back to the work of Gruenberg \cite{Gruenberg1953}.
It has a positive answer for $n\leq 4$, the case $n = 4$ was a significant achievement by Havas and Vaughan-Lee \cite{HavasVaughanLee20054Engelgroups} in 2005. The proof makes substantial use of computer algebra and builds on earlier work, notably \cite{traustason2gen4engel}. In the case of $\omega$-categorical $n$-Engel groups, the local nilpotency follows from local finiteness and the following classical theorem of Zorn (see e.g.  \cite[Theorem 12.3.4]{RobinsonZornref}).
\begin{fact}(Zorn) 
    Any finite $n$-Engel group is nilpotent.
\end{fact}

Of course, any $\omega$-categorical $p$-group is locally nilpotent by local finiteness.

For small values of $n$ and $p$, the situation is given by the following table, where the intersection of an ``$n$-Engel" row and a ``$p$" column is -if it exists- a bound on the nilpotency class of $n$-Engel $p$-groups. If it is $\infty$, it means that those are in general non-nilpotent.

\begin{center}
\begin{tblr}{
 cells={c},
 vline{1-2,7} = {-}{},
 hline{1-2,6} = {-}{}
}
$p=$ & $2$ & $3$ & $5$ & $7$ & $>7$ \\
$2$-Engel & 2 & $3$ & $2$ & $2$ & $2$ \\
$3$-Engel & $\infty$ & $4$ & \hspace{-.5em} \tikzmark{left} $\infty$ & $4$ & $4$ \\
$4$-Engel & $\infty$ & $\infty$ & $\infty$ & $7$ & $7$ \\
$5$-Engel & $\infty$ & $\infty$ & $\infty$ & $\infty$ & $10$
\end{tblr}
\end{center}

Solvability is known in certain cases.

\begin{fact}\cite{Heineken1961, AbdollahiTraustason2002}
\begin{itemize}
    \item Every $3$-Engel $2$-group is solvable.
    \item Every $4$-Engel $3$-group is solvable.
\end{itemize} 
\end{fact}

In the $\omega$-categorical case, an early theorem of Wilson give an equivalence between solvability and nilpotent.
\begin{fact}(J. S. Wilson, \cite{wilson1981})\label{fact:wilson}
    Let $G$ be an $\omega$-categorical locally nilpotent group. If $G$ is solvable, then $G$ is nilpotent.
\end{fact}

By putting the above together, we have the following.

\begin{corollary}\label{cor:3engel2and3}
Let $G$ be an $\omega$-categorical $3$-Engel $2$-group or $4$-Engel $3$-group, then $G$ is nilpotent. 
\end{corollary}

The interesting case is the case of $3$-Engel and $4$-Engel $5$-groups, which will occupy the rest of this section.

\subsection{An exceptional instance of the Lazard correspondence} \label{subsec:lazard}
A main reference for this section is \cite{Khukhro1998}, see also \cite{Cicaloetal}. The following infinite expression is the Baker-Campbell-Hausdorff (BCH) formula:

\begin{align*}
f(x,y)=x+y+\frac{1}{2}[x,y]
&-\frac{1}{12}[x,y,x]
+\frac{1}{12}[x,y,y]
-\frac{1}{24}[x,y,x,y]
+\frac{1}{720}[x,y,x,x,x]
-\frac{1}{360}[x,y,x,x,y]
\\
&+\frac{1}{120}[x,y,x,y,x]
-\frac{1}{120}[x,y,x,y,y]
+\frac{1}{360}[x,y,y,y,x]
-\frac{1}{720}[x,y,y,y,y]
+\frac{1}{1440}[x,y,x,x,x,y]
+\cdots .
\end{align*}
It admits ``inverse" formulas, given by the following two expressions

\begin{align*}
h_1(x,y)=
xy[x,y]^{-\frac{1}{2}}
[x,y,x]^{\frac{1}{12}}
&[x,y,y]^{-\frac{1}{12}}
[x,y,x,x]^{-\frac{1}{24}}
[x,y,y,y]^{\frac{1}{24}}
\\
&\cdot[x,y,x,x,x]^{\frac{19}{720}}
[x,y,x,x,y]^{-\frac{23}{720}}
[x,y,x,y,x]^{\frac{1}{30}}
[x,y,x,y,y]^{\frac{1}{20}}
[x,y,y,y,x]^{-\frac{37}{720}}
[x,y,y,y,y]^{-\frac{19}{720}}
\cdots
\end{align*}

\begin{align*}
h_2(x,y)=
[x,y]
[x,y,x]^{-\frac{1}{2}}
[x,y,y]^{-\frac{1}{2}}
\cdot&
[x,y,x,x]^{\frac{1}{3}}
[x,y,x,y]^{\frac{1}{4}}
[x,y,y,y]^{\frac{1}{3}}
\\
&\cdot
[x,y,x,x,x]^{-\frac{1}{4}}
[x,y,x,x,y]^{-\frac{1}{4}}
[x,y,x,y,x]^{\frac{1}{12}}
[x,y,x,y,y]^{-\frac{1}{6}}
[x,y,y,y,y]^{-\frac{1}{4}}
\cdots
\end{align*}

The BCH and its inverse are infinite polynomial expressions, respectively in the language of Lie rings and of groups. In an appropriate setting, those can be used to define a group in the case of BCH and a Lie ring in the case of the inverse BCH. When looking at the formulas, two obstructions occur. The first is that those expressions are infinite, which is usually compensated by having a nilpotency assumption. The second obstruction concerns the coefficients/exponents, which suggest a need to divide, or take roots. A first setting where those obstructions disappear is in torsion-free divisible nilpotent groups. This is the setting of the \textit{Mal'cev} correspondence.

The setting of interest here is that of $p$-groups. A fact that is slightly apparent in the formula above, say in the inverse BCH, is that the primes divisors of the denominator of the exponent of a bracket are always smaller than or equal to the length of the bracket. The same is true in the BCH, for the coefficients. This implies that in a $p$-group of nilpotency class $c<p$, the exponents make sense as such
groups admit unique $k$-th roots for $k$ coprime to $p$. This is the setting of the \textit{Lazard correspondence}, which can be roughly summarized as follows.
\begin{enumerate}
    \item In a $p$-group $G$ of exponent $p^n$ and nilpotency class $c<p$, the operation $x+y = h_1(x,y)$ and $[x,y] = h_2(x,y)$ defines a Lie ring structure $L_G$ on the underlying set of $G$. 
    \item In a Lie ring $L$ of characteristic $p^n$ and nilpotency class $c<p$, the operation $x*y = f(x,y)$ defines a group structure $G_L$ on the underlying set of $L$.
    \item The two operations $G\leadsto L_G$ and $L\leadsto G_L$ are inverse to each other.
\end{enumerate}
(This yields an \textit{equivalence of categories} between those two classes of structures. In particular, normal subgroups are mapped to ideals, the nilpotency classes coincide etc.)

Let us now take a step back from the previous paragraph. It is apparent that, for instance, in order to make sense of the expressions $h_1,h_2$ in a given group $G$ of exponent $p^n$, it is enough that the nilpotency class of every $2$-generated subgroup of $G$ is less than $p$. The fact that $h_1$ and $h_2$ define a Lie ring structure on the underlying set of $G$, however, would have to be checked for each one of the axioms of Lie rings. Certainly, associativity of $h_1$ 
\[h_1(h_1(x,y),z) = h_1(x,h_1(y,z))\]
or other axioms such as bilinearity or the Jacobi identity will have to be checked in three-generated subgroups of $G$. In order to be safe, assuming that the nilpotency class of $3$-generated subgroups is less than $p$ always allows to conclude that the operation $x+y = h_1(x,y)$ and $[x,y] = h_2(x,y)$ define a Lie ring on the underlying set of $G$. The converse, starting from a Lie ring of characteristic $p^n$ in which every $3$-generated subring is nilpotent of class less than $p$ and defining a group via $f(x,y)$ holds for the same reason. We have established the \textit{local} version of the Lazard correspondence, which states that the BCH formula and its inverse allows to pass between the following categories:
\begin{itemize}
    \item groups of exponent $p^n$ in which every $3$-generated subgroup is nilpotent of class $<p$,
    \item Lie rings of characteristic $p^n$ in which every $3$-generated subring is nilpotent of class $<p$.
\end{itemize}

It should be clear that, while the assumption that $2$-generated subgroups being nilpotent of class $<p$ is crucial in order to have well-defined functions $h_1$ and $h_2$, the assumption that $3$-generated subgroups are nilpotent of class $<p$ is only there for applying the ``global" version of the Lazard correspondence. 
The latter assumption has no reason to be necessary in order to define a Lie algebra, and indeed, it is not. This phenomenon will occur in $3$-Engel groups of exponent $5$ and in $4$-Engel groups of exponent $7$. The following is from \cite{Gupta_Newman_1989}.

\begin{fact}\label{fact:3Engelgroups}
    Let $G$ be a $3$-Engel group. Then every $n$-generated subgroup of $G$ is nilpotent of class at most $2n-1$. Further, the group $G^5$ generated by $5$-th powers satisfies the identity
    \[[[x_1,x_2,x_3],[x_4,x_5],x_6] = 1\]
Bounds are sharp.
\end{fact}

If $G$ is a $3$-Engel $p$-group then $3$-generated subgroups are nilpotent of class at most $5$. In particular, the local Lazard correspondence applies for $3$-Engel $p$-groups with $p\geq 7$. The case of $p = 5$ is a limit case and it turns out that the Lazard correspondence applies in this case also. We call that an instance of an \textit{exceptional Lazard correspondence}. This very surprising fact seems to have been missed by specialists in the domain.

\begin{theorem}\label{thm:exceptLazard35}
    Let $G$ be a $3$-Engel group of exponent $5$. Then the formulas
    \[x+y := h_1(x,y)=
xy[x,y]^{2}
[x,y,x]^{3}
[x,y,y]^{2}\]
    and 
    \[[x,y]_L := h_2(x,y)=
[x,y]
[x,y,x]^{2}
[x,y,y]^{2}\]
    define a $3$-Engel Lie algebra $(L_G, +,[.,.]_L)$ over $\F_5$ on the domain of $G$. Further, the formula
\[x*y := f(x,y)=x+y+\frac{1}{2}[x,y]
-\frac{1}{12}[x,y,x]
+\frac{1}{12}[x,y,y]\]
defines a group structure on the domain of $L_G$, which coincides with the group structure of $G$, i.e. $(G,\cdot) = (G,*)$.
\end{theorem}

\begin{proof}
As the nilpotency class of $2$-generated subgroups of $G$ is at most $3$, every commutator of length $4$ vanishes, hence we cut off the formula to commutators of length at most $3$, which gives
    \[x+y := h_1(x,y)=
xy[x,y]^{-\frac{1}{2}}
[x,y,x]^{\frac{1}{12}}
[x,y,y]^{-\frac{1}{12}}\]
    and 
    \[[x,y]_L := h_2(x,y)=
[x,y]
[x,y,x]^{-\frac{1}{2}}
[x,y,y]^{-\frac{1}{2}}.\]
We translate the coefficients modulo $5$, for which we have $-\frac{1}{2} = 2$, $\frac{1}{12} = 3$, $-\frac{1}{12} = 2$. Proving the theorem involves first proving that the axioms of Lie $\F_5$-algebras are satisfied by the structure $(G,+,0,[.,.]_L)$. 

The easiest way to verify this is by using a computer
algebra system to run a nilpotent quotient algorithm. We wrote such a program in GAP, using the nilpotent quotient algorithm for groups provided by the anupq package \cite{ANUPQ}. The package allows to compute a free $3$-Engel rank $3$ group of exponent $5$. Then it suffices to enter the formulas above and check the axioms. We also check that the Lie algebra thus obtained is $3$-Engel. The fact that it is sufficient to compute a $3$-generated group in order to deduce a general result is explained in Remark \ref{rk:computeralgebralocnilp}. 

To get the full theorem, a further test has to be provided. From the functions $+$ and $[.,.]_L$, one defines the binary function 
\[x*y := f(x,y)=x+y+\frac{1}{2}[x,y]_L
-\frac{1}{12}[x,y,x]_L
+\frac{1}{12}[x,y,y]_L\]
and simply checks that $* = \cdot$. Such a procedure was written and run in GAP by the author and is freely available on his website:
\begin{itemize}
    \item Program: \url{https://choum.net/~chris/GAP-Lazard/exceptional_lazard_check_Engel3}
    \item Output: \url{https://choum.net/~chris/GAP-Lazard/OUTPUT-LazardProc-Class6Engel3Prime5Exp5.txt}
\end{itemize}
The program uses a file ``BCH-file" \url{https://choum.net/~chris/GAP-Lazard/BCH-file} which gives a very long list of the first terms of the BCH formula and its inverse. This list was computed by Cicalò, de Graaf and Vaughan-Lee using the algorithm described in \cite{Cicaloetal}. This file is publicly available as part of another GAP package \cite{LieRing}, namely the file ``pols.g".
\end{proof}

\begin{remark}
    It should be noted that the above theorem does \textit{not} work for groups of higher exponent. We checked that, for instance, in exponent $25$, the sum formula is not even commutative.
\end{remark}

A classical and important theorem of Higgins \cite{higginsEngel} states that if $p>n$, then any $n$-Engel Lie algebra over $\F_p$ is solvable if and only if it is nilpotent. This theorem has no analogue in the $n$-Engel $p$-group context, and the exceptional Lazard correspondence provide such an analogue.

\begin{corollary}\label{cor:higginsforgroups}
    Let $G$ be a $3$-Engel group of exponent $5$. Then $G$ is solvable if and only if $G$ is nilpotent.
\end{corollary}

\begin{proof}
    We first apply the exceptional Lazard correspondence on a $3$-Engel solvable group of exponent five $G$ in order to obtain a Lie algebra structure $L$ on the domain of $G$. Then $L$ is a $3$-Engel Lie $\F_5$-algebra, and is solvable since $G$ is. Applying Higgins \cite{higginsEngel}, $L$ is nilpotent and therefore $G$ is nilpotent.
\end{proof}

We now turn to another case of the exceptional Lazard correspondence. We get the following bounds from the work of Havas and Vaughan-Lee \cite{HavasVaughanLee20054Engelgroups} where they prove that $4$-Engel groups are locally nilpotent.

\begin{fact}
Let $G$ be a $4$-Engel $p$-group with $p>5$. Then $G$ is nilpotent of class at most $7$. If $G$ is two-generated, then $G$ has class at most $6$.
\end{fact}

In particular, the case of $4$-Engel groups of exponent $7$ is again a limit case: every $2$-generated subgroup is nilpotent of class at most $6$ and the $3$-generated subgroups have class at most $7$. Again, we get an exceptional Lazard phenomenon here.

\begin{theorem}\label{thm:except4Engel7}
    Let $G$ be a $4$-Engel group of exponent $7$. Then the formulas

\begin{align*}
x+y:= &
xy[x,y]^3
[x,y,x]^3[x,y,y]^4
[x,y,x,x]^2[x,y,y,y]^5 \\
&[x,y,x,x,x]^2[x,y,x,x,y]^2
[x,y,x,y,x]^4[x,y,x,y,y]^6 \\
&[x,y,y,y,x]^2[x,y,y,y,y]^5
[x,y,x,x,x,x]^3[x,y,x,x,x,y]^4 \\
&[x,y,x,x,y,x]^4[x,y,x,x,y,y]^2
[x,y,x,y,x,y]^6[x,y,x,y,y,x]^6 \\
&[x,y,x,y,y,y]^4
[x,y,y,y,x,y]^2
[x,y,y,y,y,y]^4 .
\end{align*}
and 
\begin{align*}
[x,y]_L := &
[x,y]
[x,y,x]^3[x,y,y]^3
[x,y,x,x]^5[x,y,x,y]^2[x,y,y,y]^5 \\
&[x,y,x,x,x]^5[x,y,x,x,y]^5
[x,y,x,y,x]^3[x,y,x,y,y]
[x,y,y,y,y]^5 \\
&[x,y,x,x,x,x]^3
[x,y,x,x,y,x]
[x,y,x,y,x,y]^2 \\
&[x,y,x,y,y,x]^2
[x,y,x,y,y,y]
[x,y,y,y,y,y]^3
\end{align*}
    define a $4$-Engel Lie algebra $(L_G, +,[.,.]_L)$ over $\F_7$ on the domain of $G$. Further, the formula
\begin{align*}
x*y := x+y+&4[x,y]+4[x,y,x]+3[x,y,y]+2[x,y,x,y]\\
&+6[x,y,x,x,x]
+2[x,y,x,x,y]+[x,y,x,y,x]
+6[x,y,x,y,y]\\
&+5[x,y,y,y,x]
+[x,y,y,y,y]+3[x,y,x,x,x,y]+[x,y,x,x,y,y]\\
&+4[x,y,x,y,x,y]+4[x,y,x,y,y,y]+6[x,y,y,y,x,y].
\end{align*}
defines a group structure on the domain of $L_G$, which coincides with the group structure of $G$, i.e. $(G,\cdot) = (G,*)$.
\end{theorem}

\begin{proof}
    The proof is exactly the same as in the proof of Theorem \ref{thm:exceptLazard35}, using the same program with different parameters. The program and output are available here: 
    \begin{itemize}
        \item Program: \url{https://choum.net/~chris/GAP-Lazard/exceptional_lazard_check_Engel4}
        \item Output: \url{https://choum.net/~chris/GAP-Lazard/OUTPUT-LazardProc-Class9Engel4Prime7Exp7.txt}.
    \end{itemize}
\end{proof}

\begin{remark}(Definability)\label{rk:lazarddefinable}
Let $G$ be a $3$-Engel group of exponent $5$ and $L_G$ be the Lie algebra obtained by applying the exceptional Lazard correspondence. Then, from a model-theoretic point of view, $G$ and $L_G$ are the same object. This follows from the fact that the operations $+$ and $[.,.]_L$ are defined by parameter-free
quantifier-free formulas. Therefore, we have $\Aut(G) = \Aut(L_G)$ and $G$ and $L_G$ have the same definable sets.
\end{remark}

\begin{remark}\label{rk:computeralgebralocnilp}
(Why can we use computer algebra?)
The fact that the varieties of $3$-Engel and $4$-Engel groups are locally nilpotent is crucially used in the proofs of Theorem \ref{thm:exceptLazard35} and \ref{thm:except4Engel7}. Let us give more details concerning this claim. Fix $k,n,e\in \N$ with $e = p^s$ and let $F = F(k,n,e)$ be the free $k$-generated $n$-Engel group of exponent $e$. By the positive solution to the restricted Burnside problem (Zelmanov, \cite{Zelmanov1990,Zelmanov1991}), there exists a minimal $i_0\in \N$ such that $\gamma_{i_0} = \gamma_{i_0+1}$ and therefore, $F/\gamma_{i_0}$ is the \textit{maximal finite quotient} of $F$ in the sense that it satisfies the following universal property.

$\bullet$
\textit{If $F$ is freely generated by $x_1,\ldots,x_k$, and $G = \vect{g_1,\ldots,g_k}$ is any \textit{finite} $n$-Engel exponent $e$ group, then the map $x_i\mapsto g_i$ extends to an epimorphism $F/\gamma_{i_0} \to G$.}

This is what allows us to exploit the nilpotent quotient algorithm (NQA). The latter allows to compute the largest \textit{finite} nilpotent quotient of a given group (with entries: the rank $r$, the exponent $e$, the Engel condition $n$), hence the group $F/\gamma_{i_0}$. Assume that we are given a locally finite $n$-Engel group $G$ of exponent $e$, then any identity of the form $P(x_1,\ldots,x_k) = 0$ computed via a computer in $F/\gamma_{i_0}$ will transfer to $P(g_1,\ldots,g_k) =0$ for any $g_1,\ldots, g_k\in G$ via the epimorphism $F/\gamma_{i_0}\to \vect{g_1,\ldots, g_k}$.
We have shown the following:

$\bullet$ \textit{Any identity computed in $F/\gamma_{i_0}$ transfers to any locally finite $n$-Engel group of exponent $e$.}

The last thing to observe is that a locally nilpotent class of groups of bounded exponent is locally finite, so $3$-Engel groups of exponent $5$ and $4$-Engel groups of exponent $7$ are locally finite. Note that $\omega$-categorical groups are also locally finite, hence for our conclusions on $\omega$-categorical groups below, we do not need the general results of Heineken.
\end{remark}

It should be understood that Theorems \ref{thm:exceptLazard35} and \ref{thm:except4Engel7} work without the use of any of the above hard facts about $3$-Engel and $4$-Engel groups, since it is just a matter of defining the right functions and checking the axioms. We do, however, use local finiteness in order to ensure an appropriate use of computer algebra, according to Remark \ref{rk:computeralgebralocnilp}. For the case of $3$-Engel $5$-groups, this is Heineken \cite{Heineken1961}, and in the case of $4$-Engel $7$-groups, it follows from the (global) nilpotency of those, whose argument is essentially Higgins \cite{higginsEngel}. Of course, if one were willing to check by hand all the computation, then Theorem \ref{thm:except4Engel7} would also have an extra application. In general, results about $n$-Engel groups can be deduced from results about $n$-Engel Lie rings by passing to the associated Lie ring $L(G) = \oplus_i \gamma_i/\gamma_{i+1}$. The problem in doing so is that defining the right variety in which the associated Lie ring lives is often quite hard, see for instance \cite{newmanvaughanlee4engelexpo5} where the variety of the associated Lie ring of a $4$-Engel group of exponent $5$ is described. In fact, the associated Lie ring of an $n$-Engel group need not be an $n$-Engel Lie ring (this is known to happen for $n\geq 5$, see \cite{VaughanLee2011LiemethodsEngelgroups}). However, when the Lazard correspondence applies, the corresponding Lie ring is obviously $n$-Engel if and only if the group is. This allows to deduce, for instance from the classical fact due to Higgins \cite{higginsEngel} that $4$-Engel Lie algebras over $\F_p$ are nilpotent for $p>5$ that $4$-Engel groups of exponent $p$ are nilpotent, for $p>5$.

\subsection{Countably categorical $3$-Engel and $4$-Engel groups}\label{sec:smalvaluesomegacat}

The main result of \cite{delbee3Engelchar5} is the following.

\begin{fact}\label{fact:3EngelChar5}
    Let $L$ be a $3$-Engel Lie algebra over a field of characteristic $5$. If there are finitely many orbits in the action of $\Aut(L)$ on $L\times L\times L\times L$, then $L$ is nilpotent.
\end{fact}

From Fact \ref{fact:3EngelChar5}, Theorem \ref{thm:exceptLazard35} and Remark \ref{rk:lazarddefinable}, we immediately get the following.

\begin{theorem}\label{thm:3expo5}
Let $G$ be a $3$-Engel group of exponent $5$ where the action of $\Aut(G)$ on $G\times G\times G\times G$ has only finitely many orbits. Then $G$ is nilpotent.
\end{theorem}

\begin{corollary}\label{cor:omegacat3Eng5grp}
Let $G$ be an $\omega$-categorical $3$-Engel $5$-group. Then $G$ is nilpotent.
\end{corollary}
\begin{proof}
By $\omega$-categoricity, $G$ admits only finitely many characteristic subgroups. Let us consider a maximal chain of characteristic subgroups:
\[G = G_1\rteq\ldots \rteq G_r\rteq G_{r+1} = 1.\]
Then, each quotient $G_i/G_{i+1}$ is characteristically simple, and by Wilson's Theorem (Fact \ref{fact:wilson}), it is enough to prove that each quotient $G_i/G_{i+1}$ is nilpotent, so we will now assume that $G$ is characteristically simple. If $G$ has exponent $5$, then the result follows directly from Theorem \ref{thm:3expo5}. 
Assume now that the exponent of $G$ is strictly larger than $5$. Then, the characteristic subgroup $G^5$ generated by the $5$-th powers of elements of $G$ is nonzero, therefore $G = G^5$. Using the second part of Fact \ref{fact:3Engelgroups}, $G$ satisfies the identity
\[[[x_1,x_2,x_3],[x_4,x_5],x_6] = 1\]
and is therefore solvable. Applying again Wilson's Theorem, we conclude that $G$ is nilpotent.
\end{proof}

\begin{corollary}\label{cor:3engelsylow}
Let $G$ be an $\omega$-categorical $3$-Engel group. Then $G$ is nilpotent.
\end{corollary}

\begin{proof}
We know that $G$ has bounded exponent, say $m = p_1^{\alpha_1}\cdots p_s^{\alpha_s}$ and is locally nilpotent, hence $G$ is the direct product of its Sylow $p$-subgroups: $G = G_{p_1}\times\cdots \times G_{p_s}$. Each $G_{p_i}$ is characteristic, hence definable, and hence nilpotent by \cite{Heineken1961} (for $p\neq 2,5$), Corollary \ref{cor:3engel2and3} (for $p = 2$) and \ref{cor:omegacat3Eng5grp} (for $p = 5$). It follows that $G$ is nilpotent.
\end{proof}

Our conclusions on $4$-Engel $5$-groups will use the following deep result of \cite{newmanvaughanlee4engelexpo5} on $4$-Engel groups of exponent $5$. 

\begin{fact}\label{fact:4engelexpo5newmanvaughanlaa}
    Let $G$ be any $4$-Engel group of exponent $5$. Then $G$ is a subdirect product of groups which are either
    \begin{itemize}
        \item center-by-($3$-Engel)
        \item nilpotent of class at most $10$.
    \end{itemize}
    In particular, $\gamma_{11}(G)$ is center-by-($3$-Engel).
\end{fact}

\begin{corollary}\label{cor:omegacat4engelexpo5}
    Let $G$ be an $\omega$-categorical $4$-Engel group of exponent $5$, then $G$ is nilpotent.
\end{corollary}
\begin{proof}
    Reasoning as in the proof of Corollary \ref{cor:omegacat3Eng5grp}, using Wilson's Theorem, we may assume that $G$ is characteristically simple. Therefore, $\gamma_{11}(G) = 1$ or $\gamma_{11}(G) = G$. In the former case, we are done, so let us assume the second case. By Fact \ref{fact:4engelexpo5newmanvaughanlaa}, $G$ is therefore center-by-($3$-Engel), so that $G/Z(G)$ is $3$-Engel. By Corollary \ref{cor:omegacat3Eng5grp}, $G/Z(G)$ is nilpotent, and therefore $G$ is also nilpotent.
\end{proof}

Finally, to extend Corollary \ref{cor:omegacat4engelexpo5} to arbitrary exponent, we use the following fact from Abdollahi and Traustason \cite{AbdollahiTraustason2002}.

\begin{fact}\label{fact:abdollahitraustason}
    Let $n,p$ be given and $r$ such that $p^{r-1}<n\leq p^r$. Let $G$ be a locally finite $n$-Engel $p$-group.
    \begin{enumerate}
        \item If $p$ is odd, then $G^{p^r}$
is nilpotent.
        \item If $p = 2$ then $(G^{2^r})^2$ is nilpotent.
    \end{enumerate}
\end{fact}

\begin{corollary}
    Let $G$ be an $\omega$-categorical $4$-Engel $5$-group. Then $G$ is nilpotent.
\end{corollary}
\begin{proof}
Reasoning as in the proof of Corollary \ref{cor:omegacat3Eng5grp}, using Wilson's Theorem, we may assume that $G$ is characteristically simple. As $\omega$-categorical groups are locally finite, we may apply Fact \ref{fact:abdollahitraustason} with $p = 5$ and $n = 4$, which imposes $r = 1$. The conclusion is that $G^5$ is nilpotent. As $G$ is characteristically simple, either $G^5 = 1$ in which case we conclude by Corollary \ref{cor:omegacat4engelexpo5} or $G^5 = G$, which yields that $G$ is nilpotent.
\end{proof}

\begin{corollary}
    Let $G$ be an $\omega$-categorical $4$-Engel group of odd exponent. Then $G$ is nilpotent.
\end{corollary}

\begin{proof}
As in the proof of Corollary \ref{cor:3engelsylow}, since the exponent has no $2$-factor.
\end{proof}

\bibliographystyle{plain}
\bibliography{biblio.bib}{}

\end{document}

%% file: preamble.tex
\usepackage{anysize}

\usepackage{lineno}
\linenumbers

\usepackage{amssymb,latexsym}
\usepackage{amsmath,amsthm}
\usepackage{amsfonts,mathrsfs}
\usepackage{mathtools}

\usepackage{mathptmx} 

\usepackage{halloweenmath}

\usepackage{multicol}

\usepackage{tikz-cd}

\usepackage{todonotes}

\usepackage[all]{xy}
\usepackage{graphicx}

\usepackage{xcolor,url}
\usepackage{hyperref}
\hypersetup{
    colorlinks,
    citecolor=airforceblue,
    filecolor=airforceblue,
    linkcolor=airforceblue,
    urlcolor=airforceblue
}

\usepackage{tabularray}

\usepackage{phonetic}
\usepackage{pb-diagram}
\usepackage{lscape}
\usepackage{marvosym}
\usepackage{multicol}

\usepackage{subcaption}

\usepackage{enumitem}

\usepackage{ytableau}

\usepackage{listings}
\usepackage{color}

\definecolor{dkgreen}{rgb}{0,0.6,0}
\definecolor{gray}{rgb}{0.5,0.5,0.5}
\definecolor{mauve}{rgb}{0.58,0,0.82}

\usepackage{forest}

\usepackage{tabularray}
\UseTblrLibrary{diagbox}

\usepackage{wrapfig}

\usepackage{tikz}
\usetikzlibrary{calc}
\usetikzlibrary{cd}

\usepackage{etoolbox}
\patchcmd{\section}{\normalfont}{\normalfont\color{airforceblue}}{}{}
\patchcmd{\subsection}{\normalfont}{\normalfont\color{airforceblue}}{}{}

\theoremstyle{plain}
\newtheorem{theorem}{Theorem}[section]

\newtheorem{corollary}[theorem]{Corollary}

\newtheorem{fact}[theorem]{Fact}
\newtheorem{conjecture}[theorem]{Conjecture}
\newtheorem*{theorem*}{Theorem}
\newtheorem*{corollary*}{Corollary}
\newtheorem*{conjecture*}{Conjecture}

\theoremstyle{definition}

\newtheorem{question}[theorem]{Question}

\theoremstyle{remark}
\newtheorem{remark}[theorem]{Remark}
\newtheorem{claim}{Claim}

\newcommand\N{\mathbb{N}}

\newcommand\F{\mathbb{F}}

\newcommand{\End}{\mathrm{End}}

\def\seq{\subseteq}

\def\rteq{\trianglerighteq}

\newcommand{\set}[1]{\{ {#1} \}}
\newcommand{\vect}[1]{\langle {#1} \rangle}

\DeclareMathOperator{\Aut}{Aut}

\definecolor{airforceblue}{rgb}{0.36, 0.54, 0.66}

\newcommand{\afbf}[1]{\textcolor{airforceblue}{\textbf{#1}}}
  
\newcommand{\circleblue}[1]{%
  \raisebox{0.5ex}{\tikz[baseline]{\node[draw=complBlue, circle, inner sep=1pt] {#1};}}%
}

\newcommand{\tikzmark}[1]{\tikz[overlay,remember picture] \node (#1) {};}

\definecolor{complBlue}{RGB}{0, 0, 139}
\definecolor{crimsonred}{RGB}{220,20,60}

\def\Ind{\setbox0=\hbox{$x$}\kern\wd0\hbox to 0pt{\hss$\mid$\hss}
\lower.9\ht0\hbox to 0pt{\hss$\smile$\hss}\kern\wd0}
\def\Notind{\setbox0=\hbox{$x$}\kern\wd0\hbox to 0pt{\mathchardef
\nn=12854\hss$\nn$\kern1.4\wd0\hss}\hbox to
0pt{\hss$\mid$\hss}\lower.9\ht0 \hbox to 0pt{\hss$\smile$\hss}\kern\wd0}

\def\indi#1{\mathop{\ \ \hbox to 0ex{\hss$\vert^{\hbox to 0ex{$\scriptstyle#1$\hss}}$\hss}
\lower1ex\hbox to 0ex{\hss$\smile$\hss}\ \ }}

\def\nindi#1{\mathop{\ \ \hbox to 0ex{\hss$\!\not{\vert}^{\hbox to 0ex{$\scriptstyle\,#1$\hss}}$\hss}
\lower1ex\hbox to 0ex{\hss$\smile$\hss}\ \ }}
